\documentclass[hidelinks]{amsart}
\usepackage[utf8]{inputenc}
\usepackage{doi}
\usepackage{comment}
\usepackage[english]{babel}
\usepackage{accents}

\theoremstyle{plain}
\newtheorem{thm}{Theorem}[section]
\newtheorem{prop}[thm]{Proposition}
\newtheorem{fact}[thm]{Fact}
\newtheorem{cor}[thm]{Corollary}

\newtheorem{lem}[thm]{Lemma}
\theoremstyle{definition}
\newtheorem{defn}[thm]{Definition}

\newtheorem{exa}[thm]{Example}

\newtheorem{rem}[thm]{Remark}

\newenvironment{pf}{\begin{proof}}{\end{proof}}

\def\S{$\mathsection$}

\newcommand{\N}{\mathbb N}

\newcommand{\nin}{\notin}
\newcommand{\ov}[1]{\overline{#1}}
\newcommand{\sN}{\mathbb N}

\newcommand{\en}[1]{[#1]}

\newcommand{\bv}{\symbol{124}}
\DeclareMathOperator{\dom}{dom}

\DeclareMathOperator{\num}{num}
\DeclareMathOperator{\vnum}{numv}
\DeclareMathOperator{\env}{env}
\DeclareMathOperator{\Env}{Env}
\def\theory#1{\mathrm{#1}}
\DeclareMathOperator{\PA}{\theory{PA}}
\DeclareMathOperator{\ZF}{\theory{ZF}}
\DeclareMathOperator{\el}{el}
\DeclareMathOperator{\val}{val}
\DeclareMathOperator{\sat}{Sat}
\let\true\ilmondoeblu
\DeclareMathOperator{\true}{True}

\DeclareMathOperator{\fm}{Fm}
\DeclareMathOperator{\tm}{Tm}

\DeclareMathOperator{\fexists}{EXISTS}
\DeclareMathOperator{\fand}{AND}
\DeclareMathOperator{\fnot}{NOT}
\let\feq\unpacchettodicrec
\DeclareMathOperator{\feq}{EQUALS}
\DeclareMathOperator{\f+}{PLUS}
\DeclareMathOperator{\fx}{TIMES}
\DeclareMathOperator{\fS}{SUCC}
\DeclareMathOperator{\fvar}{VAR}
\DeclareMathOperator{\fmod}{MODEL}
\DeclareMathOperator{\fval}{VAL}

\newcommand{\gnn}[1]{\overline{\ulcorner #1 \urcorner}}
\newcommand{\gn}[1]{\ulcorner #1 \urcorner}
\newcommand{\dt}[1]{\accentset{\bdot}{#1}}
\newcommand{\bdot}{{\mbox{\large\bfseries .}}}

\newcommand{\liff}{\leftrightarrow}

\def\nvdash{\not\vdash}
\def\nmodels{\not\models}

\begin{document}
	
	%\title{A model-theoretic arithmetical interpretation of provability logic}
	%\title{G\"odel's theorems and modal logic: a model-theoretic interpretation}
	%\title[Models within models in PA]{Provability logic: models within models in Peano
		\title[Provability logic: models within models in Peano Arithmetic]{Provability logic: models within models in Peano
			Arithmetic
		%	\footnotetext{Partially supported by the Italian research project PRIN 2017, ``Mathematical logic: models, sets, computability'', Prot. 2017NWTM8RPRIN.}
		}
		
\author{Alessandro Berarducci and Marcello Mamino}

\thanks{Partially supported by the Italian research project PRIN 2017, ``Mathematical logic: models, sets, computability'', Prot. 2017NWTM8RPRIN}
\address{Dipartimento di Matematica, Università di Pisa, Largo Bruno Pontecorvo 5, 56127 Pisa, Italy}
\email{alessandro.berarducci@unipi.it}
\address{Dipartimento di Matematica, Università di Pisa, Largo Bruno Pontecorvo 5, 56127 Pisa, Italy}
\email{marcello.mamino@unipi.it}

		%\date{27 Aug. 2021}

		\begin{abstract}In 1994 Jech gave a model-theoretic proof of G\"odel's second incompleteness theorem for Zermelo-Fraenkel set theory in the following form: $\ZF$ does not prove that $\ZF$ has a model. Kotlarski showed that Jech's proof can be adapted to Peano Arithmetic with the role of models being taken by complete consistent extensions. 
			In this note we take another step in the direction of replacing proof-theoretic by model-theoretic arguments. 	We show, without the need of formalizing the proof of the completeness theorem within $\PA$, that the existence of a model of $\PA$ of complexity $\Sigma^0_2$  is independent of $\PA$, 
			where a model is identified with the set of formulas with parameters which hold in the model. Our approach is based on a new interpretation of the provability logic of Peano Arithmetic where $\Box \phi$  is defined as the formalization of ``$\phi$ is true in every $\Sigma^0_2$-model''.
\end{abstract}		
		
		\maketitle
		
%		\vfill\pagebreak
		
		\begin{minipage}{\textwidth}
			\tableofcontents
		\end{minipage}
		\vskip2\baselineskip
		
		\section{Introduction}

The precise statement of G\"odel's second incompleteness theorem,
informally that $\PA$ cannot prove its own consistency,
depends upon the choice of an arithmetization of the sentence
``$\PA$ is consistent''. G\"odel, sketching the proof in his seminal 1930
paper~\cite{Godel31}, elected to formalize consistency as syntactic consistency.
This is by no means the only reasonable choice, as demonstrated by Thomas
Jech's remarkably short proof~\cite{Jech1994} of a model
theorethic version of the theorem for~$\ZF$. Namely that $\ZF$ cannot
prove that~$\ZF$ has a model. For arithmetic,
Jech shows how to transfer his $\ZF$ argument to~$\PA$ by means of a
conservativity result~\cite[Remark 2]{Jech1994}. Then, work by Kotlarski
adapts Jech's technique~\cite[\S 3.7]{Kotlarski2019} to obtain a direct
proof: the idea is to replace models with complete theories and use the
Hilbert-Bernays arithmetized completeness theorem.
In this note, we take another step in the direction of
replacing proof-theoretic by model-theoretic arguments: we will intend
consistency to mean that $\PA$ has models of arithmetic complexity
$\Sigma^0_2$.

Taking advantage of the fact that $\PA$ has partial truth predicates for formulas of bounded complexity, we show that the existence of a model of $\PA$ of complexity $\Sigma^0_2$  is independent of $\PA$ (Theorem \ref{thm:main}), where a model is identified with the set of formulas with parameters which are true in the model. 
The presence of parameters is what makes it possible to express Tarski's
truth conditions and do away with the arithmetized completeness theorem, as
well as any formalized notion of syntactic consistency.		
For the reader that might be interested
in comparing our approach to other proofs of G\"odel's incompleteness
theorems, we may suggest~\cite{Kotlarski2004,Kotlarski2019,Kaye1991}. 

		In our approach, we first define a $\Pi^0_3$~predicate $\fmod(x)$ expressing the fact that $x$ is a code for a $\Sigma^0_2$-model of $\PA$. We then consider an arithmetical interpretation of modal logic where $\Box \phi$ formalizes the fact that the formula $\phi$ holds in every $\Sigma^0_2$-model of $\PA$. The formula $\Box \phi$ is in fact provably equivalent to the $\Sigma^0_1$ formalization of the provability predicate ``$\PA\vdash \phi$'', but since in our formalization we want to avoid the syntactic notion of provability, we are not going to use this fact. Thus, on the face of it, $\Box \phi$ has complexity $\Pi^0_4$.  Under our interpretation of the modal operator, $\lnot \Box \perp$ says that there is a $\Sigma^0_2$-model of $\PA$, and we will prove that this statement is independent of $\PA$ reasoning as follows. The crucial step is to verify L\"ob's derivability conditions \cite{Lob1995} for our intepretation of the modal operator $\Box$, i.e.\ we need to prove:
		\begin{enumerate}
			\item $\PA \vdash \phi \implies \PA \vdash \Box \phi$ 
			\item $\PA \vdash \Box \phi \to \Box \Box \phi$ 
			\item $\PA \vdash \Box (\phi \to \psi) \to (\Box \phi \to \Box \psi)$ 
		\end{enumerate}
		Here and throughout the paper we write $\PA\vdash \theta$ to mean that $\theta$ is true in any model of $\PA$ (and we write $M\models T$ to mean that $M$ is a model of $T$). 
		The modal counterparts of 1.--3. form the basis of the so called ``provability logic''  \cite{Solovay1976,Boolos1994}. From 1.--3. and the fixed point theorem one can derive $\PA\vdash \Box(\Box \phi\to \phi) \to \Box \phi$, whose modal counterpart is also an axiom of provability logic, see for instance \cite{Verbrugge2017}. 
		
		Under our interpretation, the proof of
		3. is straightforward. 
		To prove 1. suppose there is a model $X$ of $\PA$ where $\Box \phi$ fails. We need to find a model $Z\models \PA$ where $\phi$ fails. We can assume that $X$ is countable and has domain $\N$. By definition there is $y\in X$ such that $X \models \fmod(y)$ and $X\models \text{``}y\models \lnot \phi\text{''}$, namely $X$ thinks that $y$ is a code of a $\Sigma^0_2$-model where $\phi$ fails. Given $X$ and $y$ we are able to construct a model $Z\models \PA$ (with domain $\N$) which satisfies exactly those formulas with parameters $\varphi[s]$ such that $X \models \text{``}y \models \varphi[s]\text{''}$.  In particular $Z\models \lnot \phi$, thus concluding the proof of 1.
		
		Point 2. is the aritmetization of 1., namely we show that there is a function $x,y\mapsto {}^xy$ (of complexity $\Pi^0_3$) which maps, provably in $\PA$, a code $x$ of a $\Sigma^0_2$-model $X$ and a $y$ such that $X\models \fmod(y)$, into a code of a $\Sigma^0_2$-model $Z$ as above (the most delicate part is the mechanism to handle non-standard formulas with a non-standard number of parameters). 
		
		Granted the derivability conditions, we obtain the unprovability of $\lnot \Box \perp$ by  standard methods: we define $G$ such that $\PA\vdash G \liff \lnot \Box G$ we show that $G$ is unprovable and equivalent to $\lnot \Box \perp$.
		Finally, we show 
		\begin{enumerate}
			\item[4.]  $\N \models \Box \phi \implies \PA \vdash \phi$
		\end{enumerate}
		(the opposite direction follows from 1.) and we deduce that the negation of $G$ is also unprovable, hence $\lnot \Box \perp$ is independent of $\PA$. This means that the existence of a model of complexity $\Sigma^0_2$ is independent of $\PA$. 
		
		For the proof of 4. suppose that $\PA\nvdash \phi$. Then there is a $\Sigma^0_2$-model $M$ of $\PA$ where $\phi$ fails (for a model-theoretic proof of this fact see Fact \ref{fact:kleene}). A code $m\in \N$ of $M$ withnesses the fact that $\N\nmodels \Box \phi$. 
		
		\section{Primitive recursive functions} \label{sect:natural}
		The language of $\PA$ has function symbols $0,S,+,\cdot$ for zero, successor, addition, and multiplication. The axioms of $\PA$ are those of Robinson's arithmetic $\theory Q$ plus the first-order induction scheme. The standard model of $\PA$ is the set $\N$ of natural numbers with the usual interpretation of the symbols.  
		
		If $t$ is a closed term of $\PA$ and $M$ is a model of $\PA$, let $t^M\in M$ be the value of $t$ in $M$. 
		If $n\in \N$, let $\ov n = S^n(0)$ be the numeral for $n$. In the standard model $\N$ the value of $\ov n$ is $n$.
		If $f:\N\to \N$ is a primitive recursive function then (using G\"odel's $\beta$-function)  $f$ can be represented by a $\Sigma^0_1$-formula $\psi(x,y)$ of $\PA$ in such a way that, forall $m,n\in \N$ we have:
		\begin{enumerate}
			\item  $f(m) = n \implies \PA \vdash \psi(\ov m, \ov n)$
			\item  $f(m) \neq n \implies \PA \vdash \lnot \psi(\ov m, \ov n)$
			\item $\PA \vdash \forall x \exists ! y \psi(x,y)$ 
		\end{enumerate}
		and similarly for $n$-ary functions. In the above situation we shall often write $f(x) = y$ as shorthand for the formula $\psi(x,y)$.  
		Given a model $M$ of $\PA$, with our notational conventions, we have 
		$$f(m) = n \iff M \models f(\ov m) = \ov n.$$ 
		We recall that an element of $M$ is standard if it is the value of some numeral, i.e. it is of the form ${\ov n}^M$ for some $n\in \N$.   If we identify $n\in \N$ with ${\ov n}^M\in M$, then 1.--3. say that $\psi$ defines an extension of $f:\N\to \N$ to a function $f:M\to M$.  
		In general $\psi$ can be chosen to satisfy additional
properties which depend on the way $f$ is presented as a primitive
recursive function. Consider for instance the function $f(n) = 2^n$
presented via the functional equations $2^0 = 1$ and $2^{n+1} = 2^n 2$.
Then $\psi$ can be chosen in such a way that $\PA \vdash \forall x (
2^{x+1} = 2^x \cdot 2)$, where $2^x$ is defined within $\PA$ as the unique
$y$ such that $\psi(x,y)$. With this choice of $\psi$, in any model $M$ of
$\PA$, the functional equation $2^{x+1} = 2^x 2$ continues to hold for
non-standard values of $x$, thus $\psi(x,y)$ determines (by the induction
scheme of $\PA$) a unique definable extension of the function $n\in
\N\mapsto 2^n\in \N$ to the non-standard elements. In general, two different
presentations of the same primitive recursive function determine different
extensions to the non-standard elements, unless $\PA$ is able to show that
the two representations are equivalent. A representation is natural if
$\PA$ proves the validity of the same functional equations that are used
in the presentation of the function in the metatheory. 
		%Admittedly this is not precise, as we did not specify what we mean by ``formal version''. The vagueness is however mitigated by the fact that $\PA$ is a strong theory. Minor differences in the formalization, which to a human observer appear insignificant, will also appear insignificant to $\PA$: it is indeed an empirical fact that $\PA$ proves the equivalence of different formalizations based on the same ideas. 
		We shall always assume that the primitive recursive functions we consider are represented in $\PA$ in a natural way. 
		Given a formula $\phi(x)$ of $\PA$ and a primitive
recursive function $f$, we will feel free to write $\phi(f(x))$ as a
short-hand for the formula $\exists y (f(x) = y \land \phi(y))$, where
``$f(x) = y$'' stands for the formula $\psi(x,y)$ that we have chosen to
represent $f$ inside $\PA$. So, for instance, it makes sense to write
$\phi(2^x)$ although the language of $\PA$ does not have a symbol for the
exponential function. Using similar conventions, we may act as if the
language of $\PA$ had been enriched with a symbol for each primitive
recursive function, or, more precisely, for each primitive recursive
presentation of a function. 
		%Let us note, if this is not already clear from the above discussion, that different primitive recursive definitions of the same primitive recursive functions may not be provably equivalent within $\PA$. Choosing a natural representation of primitive recursive function $f$ essentially involves choosing a particular primitive recursive definition of $f$.  
		
		%We write $x\in \fm$ for $\fm(x)$ and $x\in \tm$ for $\tm(x)$. Note that here $x$ is just a variable of $\PA$. 
		
		We fix an effective G\"odel numbering of terms and
formulas of $\PA$ and we write $\gn{\phi}\in \N$ for the G\"odel number of $\phi$. 
		%We then define $\gnn{\phi}$ (with slightly bigger corners) as the numeral of the G\"odel number of $\phi$. So by definition $\gnn{\phi} = \ov{\gn{\phi}}$. 
		In the next section we will introduce various primitive recursive functions involved in the formalization of syntactic notion. 
		%, for instance a function $\fand:\N\times \N\to \N$ such that $\fand (\gn{\phi},\gn{\psi}) = \gn{\phi \land \psi}$. Formalizing the definition in $\PA$ we then have $\PA \vdash \fand (\gnn{\phi},\gnn{\psi}) = \gnn{\phi \land \psi}$.
		We use $x_0, x_1, x_2, \ldots$ as formal variables of $\PA$, but we also use other letters (such as $x,y,z,t$) as metavariables standing for formal variables.
		% The G\"odel number of a variable $x_i$ will be a primitive recursive function of its index $i$. 
		
		\section{Arithmetization}
		The content of this section is entirely standard, but we include it to fix the notations. 
		
		\begin{prop}
			There are primitive recursive functions $\fS$,
$\f+$, $\fx$, $\fvar$, which are increasing in both arguments, such that: 
			\begin{itemize}
				\item $\fS(\gn{t}) = \gn{S(t)}$
				\item $\f+ (\gn{t_1},\gn{t_2}) = \gn{t_1+t_2}$ 
				\item $\fx (\gn{t_1},\gn{t_2}) = \gn{t_1 \cdot t_2}$
				\item $\fvar(i) = \gn{x_i}$ 
			\end{itemize}
			where $t,t_1,t_2$ are terms and $i\in \N$. 
		\end{prop}
		The above functions can be naturally represented in $\PA$ by $\Sigma^0_1$-formulas, so they have a natural extension (denoted by the same names) to non-standard models of $\PA$. 
		By formalizing the recursive definition of the class of terms inside $\PA$ we obtain: 
		
		\begin{prop}
			There is a formula $\tm(x)\in \Sigma^0_1$ such that $\PA$ proves that, for all $x$, $\tm(x)$ holds if and only if one and only one of the following alternatives holds:
			\begin{itemize}
				\item  $\exists i \; \; x = \fvar(i)$
				\item $x = \gnn{0}$
				\item $\exists a \; \tm (a) \land x = \fS(a)$
				\item $\exists a,b \; \tm (a) \land \tm(b) \land x = \f+(a,b)$
				\item $\exists a,b \; \tm (a) \land \tm(b) \land x = \fx(a,b)$ 
			\end{itemize}
			Since the class of (codes of) terms is a primitive recursive, under the natural formalization both $\tm(x)$ and its negation are equivalent, in $\PA$, to $\Sigma^0_1$-formulas. 
		\end{prop}
		
		\begin{cor}
			For every term $t$ of $\PA$, $\PA \vdash \tm(\gnn{t})$.
		\end{cor}
		
		We have analogous propositions for the codes of formulas.  
		\begin{prop} There are primitive recursive functions
$\fnot$, $\fand$, $\fexists$, $\feq$, which are increasing in both
arguments, such that:
			\begin{itemize}
				\item $\fnot (\gn{\phi}) = \gn{\lnot \phi}$
				\item $\fand(\gn{\phi},\gn{\psi}) = \gn{\phi \land \psi}$
				\item $\fexists(i,\gn{\phi}) = \gn{\exists x_i \phi}$
				\item $\feq(\gn{t_1},\gn{t_2}) = \gn{t_1=t_2}$
			\end{itemize}
			where $\phi,\psi$ are formulas, $t_1,t_2$ are terms, and $i\in \N$. 
		\end{prop}
		
		The above functions can be naturally represented in $\PA$ by $\Sigma^0_1$-formulas, so they have a natural extension (denoted by the same names) to non-standard models of $\PA$.
		
		\begin{prop}
			There is a formula $\fm(x)\in \Sigma^0_1$ such that $\PA$ proves that, for all $x$, $\fm(x)$ holds if and only if one and only one of the following alternatives holds: 
			\begin{itemize}
				\item $\exists a,b \; \tm(a) \land \tm(b) \land x = \feq(a,b)$
				\item $\exists \phi \; \fm(\phi) \land x = \fnot(\phi)$
				\item $\exists \phi, \psi \; \fm(\phi) \land \fm(\psi) \land x = \fand(\phi,\psi)$ 
				\item $\exists i, \phi \; \fm (\phi) \land x = \fexists(i,\phi)$
			\end{itemize}
			Since the class of (codes of) formulas is primitive recursive, under the natural formalization both $\fm(x)$ and its negation are equivalent, in $\PA$, to $\Sigma^0_1$-formulas. 
		\end{prop}
		
		\begin{cor}
			For every formula $\phi$, $\PA\vdash \fm(\gnn{\phi})$. 
		\end{cor}
		
		%	If $n\in \N$, we write $\PA \vdash \varphi(n)$ if $\PA \vdash \varphi(S^n(0))$ where $S^n(0)$ is the numeral for $n$. 
		
		\begin{defn}\label{defn:fm}
			If $M$ is a model of $\PA$ and $\phi\in M$ is such that $M \models \fm(\phi)$, we will say that $\phi$ is an arithmetized formula in the model $M$. Similarly, an arithmetized term of $M$ is an element $a\in M$ such that $M \models \tm(a)$. 
		\end{defn}
		
		If $\psi$ is a formula of $\PA$ in the metatheory, then $\gnn{\psi}^M$ is an arithmetized formula of $M$, but if $M$ is non-standard there are arithmetized formulas which are not of this form. Similarly, if $t$ is a term of $\PA$, then $\gnn{t}^M$ is a arithmetized term of $M$, and if $M$ is non-standard it will also contain non-standard arithmetized terms.

		\section{$\Sigma^0_2$-models}
		In this section we define a $\Sigma^0_2$-model as a model $M$ with domain $\N$ such that the set of formulas with parameters which are true in the model is $\Sigma^0_2$-definable (so the standard model $(\N,0,S,+,\cdot)$ is not $\Sigma^0_2$-definable). We proceed below with the formal definitions. 
		
		An infinite sequence of natural numbers $(a_n)_n$ is finitely supported if there is $k\in \N$ such that $a_n = 0$ for all $n\geq k$. There is a bijection between natural numbers and finitely supported sequences of natural numbers: it suffices to map $s\in \N$ to the sequence of the exponents appearing in the prime factorization $\Pi_k p_k^{a_k}$ of $s+1$ (where  $p_0 = 2, p_1 = 3, p_2 = 5$ and in general $p_k$ is the $k+1$-th prime). 
		
		\begin{defn}[PA]\label{defn:el} Given $s,k$, 
			let $\el(s,k)$ be the least $a$ such that
$p_k^{a+1}$ does not divide $s+1$. According to the definition, $$s+1 = \Pi_k p_k^{\el(s,k)}$$ where  $\Pi_k p_k^{\el(s,k)}$ can be regarded as a finite product since all but finitely many factors are equal to $1$.  		Note that $\el(s,k)$ is a primitive recursive function of $s,k$. 
		\end{defn}

\begin{rem}[PA]\label{rem:el}
The coding of finitely supported sequences defined above is injective
\[s_1 = s_2 \;\;\leftrightarrow\;\; \forall k\; \el(s_1,k) = \el(s_2,k)\]
\end{rem}
		
		\begin{prop}[PA]\label{prop:subs}
			Given $s,a,k$, there is a unique  $t$, denoted $s[a/k]$,  such that $\el(t,i) = \el(s,i)$ for all $i\neq k$ and $\el(t,k) = a$. 
			
			Note that $s[a/k]$ is a primitive recursive function of $s,a,k$.
			% (since the universal quantifier on $i$ can be bounded by $s$). 
		\end{prop}
		
		We will consider countable models $M$ of $\PA$. We can assume that all such models have domain $\N$, but the intepretation of the function symbols $0,S,+,\cdot$ will in general differ from the standard one. 
		
		\begin{defn}\label{defn:environment}
			Let $M = (\N; \;0_M,S_M,+_M,\cdot_M)$ be a model of $\PA$ with domain $\N$.  If $\phi$ is a formula in the language of $\PA$ and $s\in \N$ we write $$M\models \phi \en{s}$$ to express the fact that $\phi$ holds in $M$ in the environment coded by $s$, i.e. the environment which, for each $i$, assigns the value $\el(s,i)$ to the variable $x_i$. For simplicity we take as a basis of logical connectives $\lnot,\land, \exists$ (negation, conjunction, existential quantification). The universal quantifier~$\forall$ and the logical connectives~$\wedge$ and~$\to$ are defined in terms of $\lnot,\land, \exists$ in the usual way. Tarski's truth conditions then take the following form:  
			\begin{itemize}
				\item $M\models (\exists x_i \phi) \en{s} \iff \; \text{ there is} \; x\in \N \; \text{such that}\; M \models \phi (s[x/i])$
				\item $M\models (\phi \land \psi) \en{s} \iff M\models \phi \en{s} \; \text{and} \; M \models \psi \en{s}$
				\item $M\models (\lnot \phi) \en{s} \iff M\nmodels \phi \en{s}$
				\item $M \models (t_1= t_2) \en{s} \iff \val(t_1,M,s) = \val(t_2,M,s)$
				%	\item $M\models (x_i = x_j) \en{s} \iff \el(s,i)$ is equal to $\el(s,j)$ 
			\end{itemize}
			where $\val(t,M,s)$ is the value of the term $t$ in the model $M$ when variables are evaluated according to $s$, namely $\val(x_i,M,s) = \el(s,i)$.
			
		\end{defn}
		If $\phi$ is closed (it has no free variables),  then the validity of a formula $\phi$ in $M$ does not depend on the environement: $M\models \phi \en{s} \iff M\models \phi\en{0}$. In this case we may write $M\models \phi$ for $M\models \phi\en{0}$. Occasionally we make use of the  connective $\perp$ standing for ``false''. Thus for every $M$ we have $M\nmodels \perp$. 
		
\begin{defn}\label{defn:complexity} Let $M$ be a model of $\PA$ with
domain $\N$. We say that $M$ is a $\Sigma^0_2$-model if the set of pairs
$(\gn{\phi},s)\in \N\times \N$ such that $M\models \phi\en{s}$ is an
arithmetical set of complexity~$\Sigma^0_2$. 

For a technical reason, which
will be clarified in the comments before Lemma \ref{lem:Env}, we assume that the constant~$0$ is interpreted
in~$M$ with the element~$0\in\N$, namely $0_M=0$. 
\end{defn}
		
		We recall that a set of natural numbers is $\Delta^0_2$ if both the set and its complement can be defined by a $\Sigma^0_2$-formula.  Notice that a $\Sigma^0_2$-model is in fact automatically $\Delta^0_2$.  We will need the following fact. 
		%The following proposition follows from the well known fact that a recursive binary tree with an infinite path has a $\Delta^0_2$-infinite path. 
		
		\begin{fact} \label{fact:kleene}
Let $T$ be a recursively axiomatized theory without finite models. If $T$
has a model, then $T$ has a model whose elementary diagram has arithmetic
complexity~$\Delta^0_2$.
		\end{fact} 
		
		Fact \ref{fact:kleene} can be easily derived from the
usual proof of the completeness theorem based on K\"onig's lemma, together
with the observation that a recursive binary tree with an infinite path
has a $\Delta^0_2$ infinite path (see \cite{Kleene1952,Shoenfield1960}).
We thank the anonymous referee for suggesting that it can also be derived
model-theoretically from Skolem's proof of the existence of countable
models as limits of finite models in \cite{Skolem1922} (see also p. 20-21
of \cite{Wang1970} and related developments in
\cite{Quinsey2019,Shelah1984}). We include a model-theoretic proof below.
We stress that Fact \ref{fact:kleene} will only be used in the metatheory,
namely we do not need to formalize its proof within $\PA$. Moreover, Fact
\ref{fact:kleene} will only be used in the proof of $\PA\nvdash
\Box\perp$, but not in the proof of $\PA\nvdash \lnot \Box \perp$.
		%Skolem’s method leads to Kripke's notion of fulfilment exploited \cite{Quinsey2019,Shelah1984}). 
		%suggesting the reference \cite{Quinsey1980} in this regard).  
		%It is worth to stress that we need Fact \ref{fact:kleene} only in the metatheory, namely we only need to know that it is true, but we do no need its proof in the construction of the formula $\Box \perp$. 
		% See \cite{Shoenfield1960} for related results and pointers to the relevant literature.
		
		\begin{proof}[Proof of Fact \ref{fact:kleene}.] 
			%We can assume that the language $L$ of $T$ is relational (possibly with equality). 
			We can assume that $T$ has a $\vec{\forall} \vec{\exists}$-axiomatization, namely it is axiomatized by formulas of the form $\forall \bar x  \exists \bar y \theta(\bar x, \bar y)$ where $\theta$ is quantifier free and $\bar x, \bar y$ are tuples of variables. We can reduce to this situation by expanding the language $L$ of $T$ with the introduction of a new predicate symbol $R_{\varphi}(\bar x)$ for each $L$-formula $\varphi(\bar x)$ together with the following axioms: 
			%axioms ensuring that $R_\varphi(\bar x)$ is equivalent to $\varphi(\bar x)$. This can be done in a $\vec{\forall} \vec{\exists}$-manner dealing with one logical connective at a time. More precisely, the new axioms have the form: 
			\begin{itemize}
				\item $R_\varphi(\bar x) \liff \varphi(\bar x)$ for each atomic $\varphi$
				\item $R_{\lnot \varphi}(\bar x) \liff \lnot R_{\varphi}(\bar x)$
				\item $R_{\alpha\land \beta}(\bar x) \liff R_{\alpha}(\bar x) \land R_{\beta}(\bar x)$
				\item $R_{\exists y \varphi}(\bar x) \liff \exists y R_\varphi(\bar x,y)$ 
				\item $R_{\forall y \varphi}(\bar x) \liff \forall y R_\varphi(\bar x,y)$
			\end{itemize}
			(with implicit universal quantifiers over $\bar x$). 
			After such a modification, we can assume that $T$ has effective elimination of quantifiers, a $\vec \forall \vec \exists$-axiomatization, and is formulated in a relational language $L$ (possibly with equality). We need to find a model of $T$ whose atomic diagram is $\Delta^0_2$ (the elementary diagram will then also be $\Delta^0_2$ because $T$ has effective elimination of quantifiers). 
			
			We will construct a $\Delta^0_2$-model of $T$ as a limit of finite models following the ideas of \cite{Skolem1922,Shelah1984} with suitable modifications to handle theories rather than single formulas. We need some definitions. 
			
			Let $S\subseteq L$ be a finite fragment of the language $L$. 
			%				Let $(\varphi_n)_{n\in \N}$ be a recursive enumeration of the axioms of $T$ and let $L_n\subseteq L$ be the finite sub-language of $L$ containing the non-logical symbols of $\varphi_0,\ldots, \varphi_n$. 
			An $(S,m)$-structure is a finite sequence of $S$-structures $\bar M = (M_0,M_1,\ldots, M_m)$ such that $M_\ell$ is a substructure of $M_{\ell + 1}$ for all $\ell < m$. Given another $(S,m)$-structure $\bar N$, we say that $\bar N$ is an $m$-substructure of $\bar M$ if $N_\ell$ is a substructure of $M_\ell$ for all $\ell \leq m$.  
			
			Let $\varphi := \forall \bar x \exists \bar y \theta$ be a closed formula, with $\theta$ quantifier free. 
			We say that $\varphi$ is a $(p,q)$-formula if the number of $\forall$-quantifiers in $\varphi$ is $p$ and the number of $\exists$-quantifiers is $q$.
			
			If $\bar M$ is a $(S,m)$-structure and $\varphi$ is a closed $(p,q)$-formula in the language $S$, we say that $\bar M$ is an $(S,m)$-model of $\varphi$, if for all $\ell < m$ and for every $a_1,\ldots,a_p \in \dom(M_\ell)$ there are $b_1, \ldots, b_q \in \dom(M_{\ell + 1})$ such that $M_{\ell+1}\models \theta(\bar a, \bar b)$. Note that a $(S,0)$-structure satisfies every closed formula. 
			
			We say that $\bar M$ is $(p,q)$-bounded if $\bv M_0 \bv = 1$ and for all $\ell<m$, 	$\bv M_{\ell+1} \bv \leq \bv M_\ell \bv + q \bv M_{\ell} \bv^{p}$. Note that if $\bar M$ is $(p,q)$-bounded, then it is $(a,b)$-bounded for all $a\geq p, b\geq q$.	
			%			Given two $(S,m)$ structures $\bar M$ and $\bar N$, we say that $\bar N$ is a substructure of $\bar M$ if $\bar N_\ell$ is a substructure of $\bar M_\ell$ for all $\ell \leq n$.  
			
			The following facts follow easily from the definitions. The idea of the proof is as in \cite[Claim 1.3]{Shelah1984} with minor adaptations. 
			\begin{enumerate}
				\item If $\varphi$ has a model, then for
every $n\in \N$ $\varphi$ has an $(S,n)$-model $\bar M$.
				\item If $\varphi = \forall \bar x \exists \bar y \theta$ is a $(p,q)$-formula with an $(S,n)$-model $\bar M$, then $\varphi$ has a $(p,q)$-bounded $n$-submodel $\bar N$. (Proof: Define $N_\ell$ by induction on $\ell$. Pick an arbitardy element $a\in M_0$ and put $N_0=\{a\}$. Given $\ell<n$, there are $\bv N_\ell \bv ^p$ possible $p$-tuples $\bar x$ from $N_\ell$. For each of them choose a $q$-tuple $\bar y$ from $M_{\ell+1}$ witnessing $\theta(\bar x, \bar y)$ and put its elements  in $N_{\ell+1}$.) 
			\end{enumerate}
			An $(S,n)$-structure $\bar N = (N_0,\ldots, N_n)$ is called initial if $N_n$ is a finite initial segment of $\N$ (we do not require that $N_\ell$ is initial for $\ell<n$). 
			%is included in $\N^2$ for each $\ell\leq n$. More precisely, we can arrange so that $N_0 = \{(0,0)\}$ and $N_\ell \setminus N_{\ell-1} = \{0,\ldots, \bv N_\ell \setminus N_{\ell-1} \bv-1\} \times \{\ell \} \subseteq \N^2$ for each $0<\ell\leq n$.
			We observe that, for fixed $S,n,p,q$, there are only finitely many $(p,q)$-bounded initial $(S,n)$-structures and that any $(p,q)$-bounded $(S,n)$-structures is isomorphic to an initial one. 
			%$(p,q)$-canonical $(S,n)$-structures. 
			
			Let $(\varphi_n)_{n\in \N}$ be a recursive enumeration of the axioms of $T$ and let $L_n$ be the language of $\varphi_0\land \ldots \land \varphi_n$ (a finite fragment of $L$). Let $a_n,b_n \in \N$ be such that $\varphi_n$ is a closed $(a_n,b_n)$-formula. Let $P:= (p_n)_{n\in \N}$ and $Q := (q_n)_{n\in \N}$ where $p_n := \max_{k\leq n} a_k$ and $q_n := \sum_{k\leq n} b_k$. Since $\varphi_0\land \ldots \land \varphi_n$ is equivalent to a $(p_n,q_n)$-formula in the language $L_n$, there is an initial $(p_n,q_n)$-bounded $(L_n,n)$-model $\bar N$ of $\varphi_0,\ldots, \varphi_n$. We call such a structure a $T_{\bv n}$-model. 
			
			We say that a $T_{\bv n+1}$-model $\bar M = (M_0,\ldots, M_{n+1})$ extends a $T_{\bv n}$-model $\bar N=(N_0,\ldots, N_n)$, if for each $\ell \leq n$, $N_\ell$ is the $L_n$-reduct of a substructure of $M_\ell$ (which is a $L_{n+1}$-structure). 
			
			We define a finitely branching forest $M_T(P,Q)$ as follows. The roots of $M_T(P,Q)$ are the $T_{\bv 0}$-models. For $n>0$, the nodes of $M_T(P,Q)$ at level $n$ are the $T_{\bv n}$-models which extend some node of $M_T(P,Q)$ at leven $n-1$. The extension relation turns $M_T(P,Q)$ into a finitely branching forest (we may make it into a finitely branching tree by adding a fictitious new root). 
			
			By induction on $n$ one can show that every $T_{\bv n}$-model is isomorphic to a node of $M_T(P,Q)$ at level $n$. Assuming that $T$ has a model, it follows that $M_T(P,Q)$ is infinite. Since moreover $M_T(P,Q)$ is recursive and finitely branching, $M_T(P,Q)$ has an infinite path of complexity $\Delta^0_2$ (just take the left-most path with respect to some natural ordering). Let $M$ be the union of the structures $M_m$ such that there is an $m$-model of the form $\bar M = (M_0,M_1,\ldots, M_m)$ in the path (the domain of $M$ is the union of the domains, and the interpretation of each relation symbol $R\in L$ is the union of its interpretations in those $M_m$ in which it is defined).  
			Then $M$ is a model of $T$ whose atomic diagram has complexity $\Delta^0_2$. 
		\end{proof}

		\section{Codes of models}
		%For the semantical version of G\"odel's theorem we need a formalization of semantics within $\PA$. This is made possible by the fact that we only deal with models of bounded complexity. 
		
		In this section we define the notion of $\Sigma^0_2$-model and show that the set of codes of $\Sigma^0_2$-models is $\Pi^0_3$-definable (Proposition \ref{prop:code-model}). This is related to the observation in \cite{Kotlarski2004} that the set of codes of consistent complete extensions of a recursively axiomatized theory is $\Pi^0_3$-definable. The difference is that our formulation does not involve the syntactic notion of consistency, which would require fixing a proof-system. 
		
		We need the fact that in $\PA$ there are $\Sigma^0_n$-truth predicates for $\Sigma^0_n$-formulas (see \cite{Hajek1993}). In particular we have: 
		%We need a truth predicate for $\Sigma^0_2$-formulas in $x_0,x_1$, possibily depending on other variables $z_0,\ldots, z_n$ which play the role of parameters. The precise conditions are expressed by the following proposition. 
		
		\begin{fact}\label{fact:sat}There is a formula $\sat_2(x_0,x_1)\in \Sigma^0_2$ such that 
			for every 
			$\psi (x_1) \in \Sigma^0_2$, $$\PA \vdash \forall x_1 \; \sat_2 (\gnn{\psi},x_1) \liff \psi(x_1);$$
		\end{fact}
		For our purposes we need a variation of $\sat_2$ which works for formulas in two variables and additional parameters as in the following corollary. 
		
		\begin{cor}\label{cor:sat}
			There is a formula $\sat(x_0,x_1,x_2) \in \Sigma^0_2$ such that for every $n\in \N$ and every formula $\psi(z_1, \ldots, z_n,x,y) \in \Sigma^0_2$, 
			$$\PA\vdash \forall a_1, \ldots, a_n, \;\exists c\; \forall x, y \;  \sat(c,x,y) \liff \psi(a_1, \ldots, a_n,x,y).$$
		\end{cor}
		The idea is that $c$ codes the predicate $\{(x,y) \mid \psi(a_1,\ldots, a_n, x,y)\}$. 
		
		\begin{pf} We make use of the predicate $\sat_2$ of Fact \ref{fact:sat} and of the coding of sequences in Definition \ref{defn:el}. For simplicity we write $(s)_i$ for $\el(s,i)$. Let $\sat(c,x,y)$ be the formula $\sat_0((c)_0, f(c,x,y))$ where $f(c,x,y)$ is the least $t$ such that:
			\begin{itemize}
				\item $(t)_0 = x$
				\item $(t)_1 = y$
				\item $\forall i>0 \; (t)_{i+1} = (c)_i$
			\end{itemize}
			Now, given $\psi$, there is a $\Sigma^0_2$-formula $\theta_\psi(t)$ such  that, in $\PA$,  $$\theta_\psi(t) \liff \psi((t)_2, \ldots, (t)_{n+1}, (t)_0,(t)_1)$$
			Reasoning in $\PA$, given $a_1, \ldots, a_n$, let $c$ be minimal such that $(c)_0= \gn{\theta_\psi}$, $(c)_1 =  a_1, \ldots, (c)_n = a_n$. 
			Then
			\begin{align*} \sat(c, x,y) &\liff \sat_2(\gnn{\theta_\psi}, f(c,x,y))\\
				& \liff \theta_\psi(f(c,x,y)) \\
				& \liff \psi(a_1, \ldots, a_n, x,y) \end{align*}
		\end{pf}
		
		\begin{defn} \label{defn:code}
			Let $M$ be a $\Sigma^0_2$-model of $\PA$  (Definition \ref{defn:complexity}). Then by definition there is a $\Sigma^0_2$-formula $\psi_M(x_0,x_1)$ such that for all formulas $\phi$ of $\PA$ and all $s\in \N$, 
			$$M\models \phi \en{s} \iff \N\models  \psi_M(\gn{\phi},s)$$ 
			Letting $m = \gn{\psi_M}$, this is equivalent to  
			$$M\models \phi \en{s} \iff \sN\models \sat (m ,\gn{\phi},s)$$ 
			where $\sN$ is the standard model of $\PA$. If the above equivalence holds for all $(\phi,s)$ we say that $m$ is a code for the model $M$. 
			%Note that from the code we can recover $M$. 
		\end{defn}
		
		Our next goal is to show that the set of codes of $\Sigma^0_2$-models is $\Pi^0_3$-definable. We want to do so avoiding any recourse to a proof-system. 
		%A related result can be found in \cite[p.\ 128]{Kotlarski2004}, but where the author observes that the set of codes of complete consistent extensions of $\PA$ is $\Pi^0_3$. 
		
		\begin{defn} We write $\iota y$ for ``the unique $y$ such that''. 
			When we write an expression like $f(x) = \iota y. P(x,y)$ we mean that $f$ is the partial function defined as follows: if there is one and only one $y$ such that $P(x,y)$, then $f(x)$ is such a $y$; in the opposite case $f(x)$ is undefined.  
		\end{defn}
		
		\begin{defn}[PA] \label{defn:total}
			Given $m$, we define partial functions $0_m, s_m,+_m,\cdot_m$ (of arity $0,1,2,2$ respectively) as follows. Fix an arbitrary $s$ (for instance $s=0$). 
			\begin{itemize}
				\item $0_m = \iota y  . \; \sat(m,\; \gnn{0=x_0},\;s[y/0] )$
				\item $S_m(a) = \iota  y .\;  \sat (m,\; \gnn{S(x_0)= x_1}, \;s[a/0,y/1] )$
				\item $a+_m b = \iota y . \; \sat (m,\; \gnn{x_0+x_1 = x_2}, \;s[a/0,b/1,y/2] )$
				\item $a \cdot_m b = \iota y . \; \sat (m,\; \gnn{x_0\cdot x_1 = x_2}, \;s[a/0,b/1,y/2] )$
			\end{itemize}
			%with the convention that $0_m,S_m(a), a+_m b, a\cdot_m b$ are undefined the corresponding $y$ does not exist. 
			We say that $m$ is total if these functions are total, i.e.\ the various $y$ always exist and are unique. Since $\sat$ is $\Sigma^0_2$, ``$m$ is total'' is a $\Pi^0_3$-definable predicate in $m$. If $m$ is total we define a function $\fval$ whose first argument satisfies the predicate $\tm(x)$ as follows:
			\begin{itemize}
				\item $\fval(\fvar(i),m,s) = \el(s,i)$ 
				\item $\fval(\gnn{0},m,s) = 0_m$ 
				\item $\fval(\fS(a),m,s) = S_m(\fval(a,m,s))$ 
				\item $\fval(\f+(a,b),m,s) = \fval(a,m,s) +_m \fval(b,m,s)$ 
				\item $\fval(\fx(a, b),m,s) = \fval(a,m,s) \cdot_m\fval(b,m,s)$ 
			\end{itemize}
			Note that $\fval$ is $\Pi^0_3$-definable. 
		\end{defn}

		\begin{defn}[PA] \label{defn:tarski}
			We write $\fmod(m)$ if $m$ is total (Definition \ref{defn:total}) and the conjunction of the universal closure of the following clauses holds,  where the variables $\phi, \psi$ are relativized to the predicate $\fm$, the variables $a,b$ are relativized to the predicate $\tm$, and the variables $i,s$ are unrestricted.
			\begin{itemize}
\item $0_m = 0$ (see Definition \ref{defn:complexity})
\item $\sat(m, \fexists (i, \phi), s) \; \liff \;  \exists x \;
\sat(m, \phi, s[x/i])$
\item $\sat(m, \fand(\phi,\psi), s)\; \liff \; \sat(m, \phi, s) \land
\sat(m, \psi, s)$
\item $\sat(m, \fnot (\phi), s) \;\liff\; \lnot \sat(m, \phi, s)$
				%		\item $m\forces \feq (\fvar(i),\fvar(j)) \en{s} \liff \el(s,i)=\el(s,j)$
\item $\sat(m, \feq (a,b), s)\; \liff\; \fval(a,m,s) = \fval(b,m,s)$
\item $\text{Ax}_{\PA}(\phi)\; \to\; \sat(m, \phi, s)$
			\end{itemize}	 
Where $\text{Ax}_{\PA}(x)$ is the natural formalization of ``$x$ is an axiom of $\PA$''. 
		\end{defn}
		
		\begin{prop}\mbox{} \label{prop:code-model}
			\begin{enumerate}
				\item 	$\fmod(m)$ is a $\Pi^0_3$-formula
in the free variable $m$.
				\item If $M$ is a $\Sigma^0_2$-model of $\PA$ and $m$ is a code for $M$ (Definition \ref{defn:code}), then $\N \models \fmod (m)$. 
				\item If $m\in \N$ and $\N \models \fmod (m)$, then there is a $\Sigma^0_2$-model $M$ such that $$M\models \phi \en{s} \iff \N\models \sat(m,\gn{\phi},s)$$ for all $\phi,s$. 
			\end{enumerate}
		\end{prop}
		If 3. holds, $M$ is the (unique) model coded by $m$. So every $\Sigma^0_2$-model has a code, but different codes may code the same model. 
		\begin{pf} Point 1. is by inspection of the definition of
$\fmod(x)$. Indeed we have already observed that the totality condition in
Definition \ref{defn:total} is $\Pi^0_3$. It is also clear that the
negative occurrence of the subformula $\exists a \; \sat(m, \phi, s[a/i])$ in Definition \ref{defn:tarski} is $\Pi^0_3$ and the other parts in the definition of $\fmod(x)$ have lower complexity. 
			
			To prove 2.  we recall that, 	
			by its very definition, $\fmod(m)$ expresses the
fact that the set $\{(\phi,s) \mid \sat(m, \phi, s)\}$ satisfies Tarski's truth conditions for arithmetized formulas (standard or non-standard). When interpreted in the standard model $\N$, we only need to consider standard arithmetized formulas and (2) follows from the assumption that $M$ is a model. 
			
			To prove 3., let $m\in \N$ be such that $\N \models \fmod(m)$. Define $M$ as the structure with domain $\N$ which interprets $0,S,+,\cdot$ as $0_m, S_m, +_m, \cdot_m$ respectively. By induction on the complexity of the formula $\phi$ we have $M\models \phi\en{s} \iff \N\models \sat(m,\gnn{\phi},s)$. 
		\end{pf}
		
		%Having formalized the notion of $\Sigma^0_2$-models. 
		
		\section{An anti-quote notation}
		
		\begin{defn}
			If $\phi$ is a formula without free variables, we write $\true(x,\gnn{\phi})$ for $\sat(x,\gnn{\phi},0)$ and observe that $$\PA\vdash \fmod(m) \to \forall s (\true(m,\gnn{\phi}) \liff \sat(m,\gnn{\phi},s)),$$ i.e.\ $\PA$ proves that the truth of a closed formula in a model does not depend on the environment.  
		\end{defn}

		\begin{defn}
			If $\psi(x_0, \ldots, x_n)$ is a formula of $\PA$,
we write \[\true(m, \gnn{\psi(\dt{a}_0, \ldots, \dt{a}_n)})\] for
$\exists
s \;
\el(s,\ov 0) = a_0 \wedge \dotsb \wedge
\el(s,\ov n) = a_n
\land  \sat(m, \gnn{\psi(x_0,
\ldots, x_n)}, s)$. 
		\end{defn}
		
		\noindent If $\fmod(m)$ holds, 
		$\sat(m, \gnn{\psi(\dt{a}_0,\ldots, \dt{a}_n)})$ formalizes the fact that $\psi$ holds in the model coded by $m$ in the environment which assigns the value $a_i$ to the variable $x_{i}$. 
		
		Intuitively $\gnn{~}$ is a quote notation and the dot is an anti-quote. If an expression appears within the scope of $\gnn{~}$ it is only its name that matters, not its value, but if we put a dot on it, it is its value that matters and not its name. The following remark will further clarify the issue. 
		
		\begin{rem} Assume $\fmod(m)$. 
			If $f$ is a primitive recursive function, there is
a difference between $\true(m, \gnn{\psi(\dt{f}(x))})$ and $\true(m,
\gnn{\psi(f(\dt{x}))})$. In the first case we evaluate $f(x)$ outside of
$m$ and we intepret $\psi(x_0)$ in $m$ in the environment $x_0\mapsto
f(x)$. In the second case we interpret the formula $\psi(f(x_0))$ in $m$
in the environment $x_0 \mapsto x$. More precisely, $\PA$ proves that if
$\fmod(m)$ holds, then:
			\begin{itemize}
				\item $\true(m, \gnn{\psi(\dt{f}(x))})\liff \exists s (\el(s,0) = f(x) \land \sat(m, \gnn{\psi(x_0)},s))$ 
				\item $\true(m, \gnn{\psi(f(\dt{x}))}) \liff \exists t (\el(t,0) = x \land \sat(m, \gnn{\psi(f(x_0))},t))$ 
			\end{itemize} 
		\end{rem}
\noindent For example, $\true(m, \gnn{s(\dt x)=\dt{s}(x)})$ might non hold
when~$x=0$.
		
		\section{Coding environments}
		Given a finitely supported sequence $a_0, a_1, \ldots a_n, \ldots \in \N$, there is some $s\in \N$ which codes the given sequence in the sense that $\el(s,k) = a_k$ for all $k\in \N$. Now let $M$ be a model of $\PA$ with domain $\N$.

		%We may then consider $a_0, \ldots, a_n$ as a sequence of elements of $M$, and as such it will be coded in $M$ by some $t$, which in general will be different from both $s$ and $\ov{s}^M$ (in $M$ the element $\ov{s}^M$ will code the sequence $\ov{a_0}^M, \ldots, \ov{a_n}^M$, which in general is different from $a_0, \ldots, a_n$).  
		The aim of this section is to construct a function $\Env$ which, given $M$
		and~$s$, produces an element $\Env(s,M)\in M$ such that for all $k\in \N$
		$$\el^M(\Env(s,M),\ov k^M) = \el(s,k)=a_k$$

		% Namely,
		%if $t = \Env(s,M)$, then for all $k,a\in\N$ such that $\N \models
		%\el(s,k) = a$, we will have~$M \models \el(t, \ov{k}) = a$. 
		
		In fact we will produce a $\Pi^0_3$-definable function $\env$ such that given $s$ and a code $m$ for a $\Sigma^0_2$-model $M$, yields $\env (s,m)=\Env(s,M)$. 
		
		To construct $\Env(s,M)$ we encounter a technical difficulty as we need $\el^M(\Env(s,M),\ov k^M) = 0$ for all large enough $k\in \N$. When $M$ is isomorphic to $\N$ this implies $0^M= 0$, which is the technical condition required in Definition~\ref{defn:complexity}. A different approach would have been to code environments by finite sequences instead of finitely supported sequences. With this encoding the assumption $0^M=0$ becomes unnecessary at the expense of complicating the definition of Tarski's semantics. 
		
		% can represent only eTo prove the existence of $t$ we need the technical convention in Definition \ref{defn:complexity} that $0_M = 0$.
		%A philosofical digression may help to clarify the construction. Given a model $M$, we may refer to an element of $M$ either by a description or by naming the element itself, with the understanding that names are rigid designators, so they denote the same object in all possible contexts. For instance it may happen that $M$ has domain $\N$, but in $M$ the term $S(0)$ denotes the element $5$. In this case we are using $5$ as a rigid designator for the element described in $M$ by $S(0)$. 
		%The role of $\Env(s,M)$ is to produce a rigid designator for an element $t$ which, inside $M$, describes (i.e. codes) a list whose $k$-th element (for $k$ standard) is $a_k=\el(s,k)$, where $\el(s,k)$ is a description, to be interpreted outside of $M$, and $a_k$ is a rigid designator for the element described by $\el(s,k)$. 
	
		\begin{lem}\label{lem:Env} Let $M$ be a model of $\PA$ with domain $\N$. 
			Given $s\in \N$, there is a unique $t$, denoted $\Env(s,M)$,  such that:
			\begin{enumerate} 
				\item  $\forall k<s \; \forall a$, $\N \models \el(s,k) = a \implies M \models \el(t, \ov{k}) = a$
				\item $M \models \forall k \geq \ov{s}\; \el(t, k) = 0$ 
			\end{enumerate}
			Note that for $k\geq s$, we have $\N \models \el(s,k) = 0$. It follows that for all $s,k\in \N$ we have $M\models \el(\Env(s,M),\ov k) = x_0$ in the environment $x_0 \mapsto \el(s,k)$, or in other words $$\el^M(\Env(s,M),\ov k^M) = \el(s,k)$$ where the superscript indicates the model where $\el$ and $\ov k$ are evaluated. 
		\end{lem}
		\begin{pf} We will prove the following more general result: for all $n$ there is a unique $t$ such that:
			\begin{enumerate}
				\item  $\forall k<n \; \forall a$, $\N \models \el(s,k) = a \implies M \models \el(t, \ov{k}) = a$
				\item $M \models \forall k \geq \ov{n}\; \el(t, k) = 0$ 
			\end{enumerate}
			Granted this, the lemma follows by taking $n=s$. To prove our claim we proceed by induction on $n$. 
			For $n= 0$, we take $t=0$. The inductive step follows from Proposition \ref{prop:subs}, which allows to modify a given coded sequence by changing any of its values.  
		\end{pf}
		
		Recalling the substitution function $s[z/k]$ from Proposition \ref{prop:subs}, the crucial property of $\Env$ is that it 
		commutes with substitutions in the sense of the following proposition.
		
		\begin{prop}\label{prop:substitutions-easy} Let $M$ be a model of $\PA$ with domain~$\N$, then for all~\hbox{$z,s,k\in\N$} 
			$$M \models e_1[z/\ov k]= e_2$$
			where $e_1 = \Env(s,M)$ and $e_2 = \Env(s[z/k],M)$. 
		\end{prop} 
		We may write the proposition more perspicuosly as $$M\models \Env(s,M)[z/\ov k] = \Env(s[z/k],M),$$
		but note that $\Env(s,M)$ and $\Env(s[z/k],M)$ are defined outside of $M$, while 
		$e_1[z/\ov k]$ depends on the intepretation of a $\Sigma^0_1$-formula inside $M$ (the formula which defines the primitive recursive substitution function in Proposition \ref{prop:subs}). 
		\begin{pf}
		It suffices to show that for all $i\in M$, $$M\models \el(\Env(s,M)[z/\ov k],i) = \el(\Env(s[z/k],M),i).$$ We distinguish three cases: 
		\begin{itemize}
		\item $i={\ov k}^M$
		\item $i = \ov x^M$ for some $x\in \N$ different from $k$
		\item $i$ is a non-standard element of $M$ 
	\end{itemize}
In the first case both sides of the equality to be proved are equal to $z$. In the second case they are both equal to $\el(s,x)$. In the third case they are both equal to $0$. 
		\end{pf}
		
	In the rest of the section we formalize Lemma \ref{lem:env} and Proposition \ref{prop:substitutions-easy} inside $\PA$. We need some definitions.

	\begin{defn}
		Let $\num: \N\to \N$ be the primitive recursive function $n\mapsto \gn{S^n(0)}$.
	\end{defn}
	We can represent $\num$ inside $\PA$, so it will make sense to apply it to non-standard elements of a model of $\PA$. 
	
	\begin{defn}[PA] \label{defn:numv} Assuming $\fmod(m)$, let $\vnum(n,m) = \fval(\num(n), m, 0)$  (the third argument of $\fval$ codes the environment, which is irrelevant in this case). 
	\end{defn}
	If $n$ is standard, then $\vnum(n,m)$ is the value of the numeral $\ov n$ in the model coded by $m$. 	
	
	We can now define a function $\env$ such that, if $M$ is a $\Sigma^0_2$-model with code $m$, then $\env(s,m) = \Env(s,M)$.
	
	\begin{lem}[PA] \label{lem:env}
		Let $m$ be such that $\fmod(m)$. Given $s$, there is a unique $t$, denoted $\env(s,m)$, such that: 
		\begin{enumerate}
			\item $\forall k <s \;  \true(m, \gnn{ \el(\dt{t}, \dt{\vnum}(k,m)) = \dt{\el}(s,k) } )$
			\item 
			$\true(m, \gnn{\forall k \geq \dt{\vnum}(s,m) \; \el(\dt{t},k) = 0}).$
		\end{enumerate}  
	Similarly to Proposition \ref{lem:Env} for all $s,k$ we have $$\true(m,\gnn{\el(\dt{\env}(s,m),\dt{\vnum}(k,m)) = \dt{\el}(s,k)}).$$
%		Moreover if $n \geq s$, then $t$ does not depend on $n$ and we call it $\env(s,m)$. 
	\end{lem}
\begin{pf}
	By formalizing the proof of Lemma \ref{lem:Env} in $\PA$. 
\end{pf}

		We can now give a formalized version of Proposition \ref{prop:substitutions-easy}. 
	
		\begin{prop}[PA]\label{prop:substitutions} $\forall m, z, k, s$, if $\fmod(m)$, then 
			$$\true(m, \gnn{\dt{\env}(s,m)[\dt{z}/\dt{\vnum}(k,m)]= \dt{\env}(s[z/k],m)})$$ 
		\end{prop}

\begin{proof}[Proof of Proposition~\ref{prop:substitutions}]
Work in $\PA$ and assume $\fmod(m)$. Given $z,k,s$ we need to show $$\true(m, \gnn{\dt{\env}(s,m)[\dt{z}/\dt{\vnum}(k,m)]= \dt{\env}(s[z/k],m)}).$$
By Remark \ref{rem:el} and the definition of~$\fmod(m)$, this is equivalent to
\[\forall i \;\true(m, \gnn{
\el(\dt{\env}(s,m)[\dt{z}/\dt{\vnum}(k,m)],\dt i)=
\el(\dt{\env}(s[z/k],m),\dt i)})\]
We distinguish three cases: 
\begin{itemize}
\item $i = \vnum(k,m)$
\item $i = \vnum(x,m)$ for some $x\neq k$
\item none of the above, namely $i$ is a non-standard element of the model coded by $m$ 
\end{itemize}

\noindent In the first, case we have%, using Lemma \ref{lem:env}, the statment to be proved is equivalent to $\true(m, \gnn{\dt{z}= \dt{\el}(s[z/k],k)})$, which is true. 
\begin{align*}
& \true(m,\gnn{\el(\dt{\env}(s,m)[\dt{z}/\dt{\vnum}(k,m)],\dt i)  =  \dt{z}})  && \text{by Proposition \ref{prop:subs}}\\
& \true(m,\gnn{\dt{z} = \dt{\el}(s[z/k],k)}) && \text{by Lemma \ref{lem:env}}\\
& \true(m,\gnn{\dt{\el}(s[z/k],k) = \el(\dt{\env}(s[z/k],m),\dt i)}) && \text{by Proposition \ref{prop:subs}}
\end{align*}
and we conclude by transitivity of the equality inside the model coded by $m$. 

\noindent In the second case,
we have 
\begin{align*}
& \true(m,\gnn{ \el(\dt{\env}(s,m)[\dt{z}/\dt{\vnum}(k,m)],\dt i)=\el(\dt{\env}(s,m),\dt i)}) & & \text{by Proposition \ref{prop:subs}}\\
& \true(m,\gnn{ \el(\dt{\env}(s,m),\dt i)=\dt{\el}(s,x) }) & & \text{by Lemma \ref{lem:env}} \\
& \true(m,\gnn{ \dt{\el}(s,x)=\dt{\el}(s[z/k],x)}) & & \text{by Proposition \ref{prop:subs}}\\
& \true(m,\gnn{ \dt{\el}(s[z/k],x)=\el(\dt{\env}(s[z/k],m),\dt i) }) & & \text{by Lemma \ref{lem:env}}
\end{align*}
and we conclude again by transitivity of the equality.
	
\noindent In the third case,
\begin{align*}
& \true(m,\gnn{\el(\dt{\env}(s,m)[\dt{z}/\dt{\vnum}(k,m)],\dt i)  = \el(\dt{\env}(s,m),\dt i)}) && \text{by Proposition \ref{prop:subs}}\\
& \true(m,\gnn{\el(\dt{\env}(s,m),\dt i) = 0}) & & \text{by Lemma \ref{lem:env}}\\
& \true(m,\gnn{0=\el(\dt{\env}(s[z/k],m),\dt i) }) & & \text{by Lemma \ref{lem:env}} 
\end{align*}
and we conclude as above.
\end{proof}

		Recalling that $s+1 = \Pi_i p_i^{\el(s,i)}$ we can illustrate the definition of $\env$ by the following example.

		\begin{exa} Let $s +1 = 2^7 3^5$ and let $M$ be a $\Sigma^0_2$-model coded by $m$. Then $\env(s,m)$ is the unique element $t$ such that $M \models x_2+1 = 2^{x_0} 3^{x_1}$ in the environment $x_0\mapsto 7, x_1 \mapsto  5,x_2 \mapsto  t$.   Note that $7$ and $5$ are not necessarily equal to ${\ov 5}^M$ and ${\ov 7}^M$, so in general 
			$M \nmodels x_2+1 = 2^73^5$ in the environment $x_2\mapsto t$. 
		\end{exa}
		
		We are now ready to prove Propositions~\ref{prop:substitutions-easy}
		and~\ref{prop:substitutions}.

		\section{A model within a model}
		
		%The following proposition follows easily from the definition. 
		
		\begin{prop} \label{prop:inmodel}Let $X$ be a model of $\PA$ with domain $\N$. Given $y\in X$ such that $X\models \fmod(y)$, there is a model $Z\models \PA$ with domain $\N$ such that
			$$Z \models \phi\en{s} \iff X \models \sat(y, \gnn{\phi},t)$$ where $t= \Env(s,X)$.
		\end{prop}
		\begin{proof}Let $X,y$ be as in the hypothesis. Let $\mathcal Z$ be the set of pairs $(\phi,s)$ such that $X\models \sat(y,\gnn{\phi},\Env(s,X))$. We need to prove that there is a model $Z$ of $\PA$ with domain $\N$ such that $Z \models \phi[s]\iff (\phi,s)\in \mathcal Z$. To this aim we need to check Tarski's truth conditions and verify that $\mathcal Z$ contains the axioms of $\PA$. The latter condition follows easily from the assumption $X\models \fmod(y)$. Let us check the truth condition for negation:
			\begin{align*}
				(\lnot \phi,s)\in \mathcal Z & \liff  X \models \sat(y, \gnn{\lnot \phi}, \Env(s,X))\\
				& \liff X \models \lnot \sat(y, \gnn{\phi},\Env(s,X))\\
				& \liff X \nmodels \sat(y,\gnn{\phi},\Env(s,X)) \\
				& \liff (\phi,s)\nin \mathcal Z
			\end{align*}
			where in the second equivalence we used the fact that $X \models \fmod(y)$. Similarly, we can verify Tarski's truth condition for the quantifier $\exists$:
			\begin{align*}
				(\exists x_k \phi, s) \in \mathcal Z & \liff  
				X \models \sat(y,\gnn{\exists x_k\phi},\Env(s,X)))\\
				& \liff X \models \exists x_0 \sat(y, \gnn{\phi},\Env(s,X)[x_0/\ov{k}])\\
				& \liff\exists z\in \N \; X \models \sat(y, \gnn{\phi},\Env(s,X)[z/\ov{k}])\\
				& \liff\exists z \in \N\; X \models \sat(y, \gnn{\phi},\Env(s[z/k],X))\\
				& \liff \exists z\in \N  \; (\phi, s[z/k])\in \mathcal Z 
			\end{align*}
			where in the fourth equivalence we used Proposition \ref{prop:substitutions-easy}. We leave the other verifications to the reader.  
		\end{proof}	
		
		In the above proposition if $X$ is a $\Sigma^0_2$-model, then $Z$ is also $\Sigma^0_2$. In the rest of the section we prove that there is a definable function which computes a code $ {}^xy$ of $Z$ given $y$ and a code $x$ for $X$. 
		
		\begin{prop}[PA] \label{prop:mod-in-mod} Given $x,y$, there is $z$ such that for all $\phi,s$, 
			\begin{align*} \sat(z, \phi, s) \quad   \iff\quad  \true(x, \; \gnn{\sat(\dt{y},\dt{\vnum}(\phi,x),\dt{\env}(s,x))})\end{align*}	
			We define ${}^xy$ as the minimal such $z$ and observe that the function $x,y\mapsto {}^xy$ is $\Pi^0_3$-definable.  
			%	where according to our conventions the last formula stands for  
			%	$$\sat(x, \; \gnn{\sat(x_0,x_1,x_2)}, \; (0,1,2) \mapsto (y, \num(\phi,x), \env(s,x))$$
		\end{prop}
		\begin{pf}
			Given $x,y$, the set $$\{(\phi,s) \mid
			\true(x, \; \gnn{\sat(\dt{y},\dt{\vnum}(\phi,x),\dt{\env}(s,x))})\}$$ is $\Sigma^0_2$-definable with parameters $x,y$, so by Corollary \ref{cor:sat} there is some $z$ which codes this set, and we take ${}^xy$ to be the minimal such $z$. It can be readily verified that $x,y\mapsto {}^xy$ is $\Pi^0_3$-definable. 
		\end{pf}

		\begin{thm}[PA]\label{thm:model-in-model}
			If $\fmod(x)$ and $\true(x, \gnn{\fmod (\dt y)})$, then $\fmod ({}^x y)$.
		\end{thm}
		\begin{pf} We need to show, inside $\PA$, that the class of all pairs $(\phi,s)$ such that $\sat({}^xy,\phi,s)$ satisfies Tarski's truth conditions and contains the arithmetized axioms of $\PA$. The latter property is easy, so we limit ourself to verify the clauses for $\lnot$ and $\exists$ in Tarski's truth conditions. 
			\begin{align*}
				\sat({}^xy, \fnot(\phi), s) & \liff  \true(x, \; \gnn{\sat(\dt{y},\dt{\vnum}(\fnot(\phi),x),\dt{\env}(s,x))})\\
				& \liff \true(x, \gnn{\lnot \sat(\dt{y}, \dt{\vnum}(\phi), \dt{\env}(s,x))}) \\
				& \liff \lnot \true(x, \gnn{\sat(\dt{y}, \dt{\vnum}(\phi), \dt{\env}(s,x))}) \\
				& \liff \lnot \sat({}^xy, \phi, s) 
			\end{align*}
			where in the second equivalence we used the fact that $\true(x,\fmod (\dt{y}))$ and in the third we used the hypothesis $\fmod(x)$. Similarly we have: 
			\begin{align*}
				\sat({}^xy, \fexists (k, \phi), s) & \liff  
				\true(x, \; \gnn{\sat(\dt{y},\dt{\vnum}(\fexists(k,\phi),x),\dt{\env}(s,x))})\\
				& \liff \true(x, \gnn{\exists x_0 \sat(\dt{y}, \dt{\vnum}(\phi,x), \dt{\env}(s,x)[x_0/\dt{\vnum}(k,x)])}) \\
				& \liff \exists z \true(x, \gnn{\sat(\dt{y}, \dt{\vnum}(\phi,x), \dt{\env}(s,x)[\dt{z}/\dt{\vnum}(k,x)])}) \\
				& \liff \exists z \true(x, \gnn{\sat(\dt{y}, \dt{\vnum}(\phi,x), \dt{\env}(s[z/k],x))}) \\
				& \liff \exists z \sat({}^xy, \phi, s[z/k]) 
			\end{align*}
			where the fourth equivalence makes use of the properties of $\env$ (Proposition \ref{prop:substitutions}). 
		\end{pf}

		%$X \models ...$ 
		
		\section{L\"ob's derivability conditions}
		
		\begin{defn} Given a closed formula of $\PA$, we let
			$\Box \phi$ be the formula $\forall x (\fmod (x) \to \true(x, \gnn{\phi})$. Note that $\Box \phi$ has complexity $\Pi^0_4$. 
		\end{defn}
		
		The first three points of the following result correspond to L\"ob's derivability conditions in \cite{Lob1995}. 
		
		\begin{thm} \label{thm:derivability}  Let $\phi, \psi$ be closed formulas of $\PA$. We have: 
			\begin{enumerate}
				\item If $\PA\vdash \phi$, then $\PA \vdash \Box \phi$ 
				\item $\PA \vdash \Box \phi \to \Box \Box \phi$ 
				\item $\PA \vdash \Box(\phi\to \psi) \to (\Box \phi \to \Box \psi)$ 
				\item $\N\models \Box \phi \implies \PA \vdash \phi$ 
			\end{enumerate}
		\end{thm}
		\begin{pf}
			(1) Suppose $\PA\nvdash \Box \phi$. Then there is a model $X\models \PA$ such that $X \models \lnot \Box \phi$. By definition this means that there is $y\in X$ such that $X\models \fmod(y)$ and $X \models \true(y, \gnn{\lnot \phi})$. By Proposition \ref{prop:inmodel} there is a model $Z\models \PA$ such that $Z\models \lnot \phi$, so $\PA\nvdash \phi$. 
			
			(2) 	We write $\Diamond \phi$ for $\lnot \Box \lnot \phi$ and observe that $\Diamond \phi$ is provably equivalent to $\exists x (\fmod (x) \land \true(x,\psi))$. The statement to be proved is equivalent to $\PA\vdash \Diamond \Diamond \phi \to \Diamond \phi$. Now $\Diamond \Diamond \phi$ says that there exist $x,y$ such that $\fmod(x)$, $\true(x, \gnn{\fmod (\dt{y})})$ and $\true(x, \gnn{\true (\dt{y},\gnn{\phi})})$.  On the other hand $\Diamond \phi$ says that there is $z$ such that $\fmod(z)$ and $\true(z, \gnn{\phi})$. To prove the implication one can take $z = {}^xy$ as defined in Proposition \ref{prop:mod-in-mod}. 
			
			(3) 		Clear from the definitions and the rules of predicate calculus,  recalling that $\Box \theta$ stands for $\forall x \; (\fmod (x) \to \true(x, \gnn{\theta}))$. 
			
			(4) Suppose $PA\nvdash \phi$. By Fact \ref{fact:kleene} there is a $\Sigma^0_2$ model $M$ satisfying $\lnot \phi$. Let $m\in \N$ be a code for such a model. Then $\sN \models \fmod(m)$ and $\N \models \true(m, \gnn{\lnot \phi})$. This is equivalent to  $\sN \models \lnot \Box \phi$.
			%(4) It suffices to prove that $\PA\cup \{\phi\} \; \text{has a model}$ if and only if $\sN \models \Diamond \phi$.
			%We argue as follows. By Fact \ref{fact:kleene} $\PA$ has a model satisfying $\phi$ if and only if $\PA$ has a $\Sigma^0_2$-model satisfying $\phi$. The latter happens if and only if there is a code $m$ for such a $\Sigma^0_2$-model, i.e.\ there is $m\in \N$ such that $\sN \models \fmod(m)$ and $\N \models \true(m, \gnn{\phi})$. This is equivalent to  $\sN \models \Diamond \phi$. 
		\end{pf}
		
		\section{An undecidable formula}
		By the diagonal lemma given a formula $\alpha(x)$ in one free variable there is a closed formula $\beta$ such that $\PA \vdash \beta \liff \alpha(\gnn{\beta})$.  Using the diagonal lemma we can define a formula $G$ which says ``I have no $\Sigma^0_2$-model'', as in the definition below.  
		
		\begin{defn}
			Let $G$ be such that $\PA \vdash G \liff \lnot \Box G$. 
		\end{defn}
		Using Theorem \ref{thm:derivability} we deduce that $G$ is undecidable and equivalent to $\lnot \Box \perp$ by the standard arguments, see for instance \cite{Boolos1994}. We give the details below. 
		\begin{lem}
			$\PA \nvdash G$. 
		\end{lem}
		\begin{pf}
			Suppose $\PA\vdash G$. Then $\PA \vdash \Box G$ (Theorem \ref{thm:derivability}). On the other hand by definition of $G$, $\PA \vdash \lnot \Box G$, contradicting the consistency of $\PA$. 
		\end{pf}
		
		\begin{lem}
			$\PA \vdash G \liff \lnot \Box \perp$. 
		\end{lem}
		\begin{pf}
			We use 1.--3. in Theorem \ref{thm:derivability}. Reason in $\PA$. 
			If $G$ holds, we get $\lnot \Box G$ by definition of $G$. Since $\perp \to G$ is a tautology we obtain $\Box \perp \to \Box G$, hence $\lnot \Box \perp$. 
			
			Now assume $\lnot G$. By definition of $G$ we get $\Box G$ and by point 2. in Theorem \ref{thm:derivability} $\Box \Box G$ follows. Moreover we have  $\Box (\Box G \liff \lnot G)$ (apply the definition of $G$ inside the $\Box$), so we get $\Box \lnot G$. Since we also have $\Box G$, we obtain $\Box \perp$.
		\end{pf}
		
		\begin{lem}
			$\PA \nvdash \lnot G$. 
		\end{lem}
		\begin{pf}
			Suppose $\PA \vdash \lnot G$. Then by definition of $G$, $\PA \vdash \Box G$, so $\N \models \Box G$ and by Theorem \ref{thm:derivability}(4) $\PA \vdash G$, contradicting the consistency of $\PA$. 
		\end{pf}
		
		We have thus obtained: 
		
		\begin{thm}\label{thm:main}
			$\lnot \Box \perp$ is independent of $\PA$, namely $\PA$ does not prove that $\PA$ has a $\Sigma^0_2$-model.  
		\end{thm}

		\subsubsection*{Acknowledgements}
		We thank the anonymous referee for his comments related to Fact \ref{fact:kleene}. 
		
%		\subsection*{Conflict of interest} On behalf of all authors, the corresponding author states that there is no conflict of interest.  
				
		\bibliographystyle{alpha}
		%\bibliography{godel.bib}
		
	\end{document}